\documentclass[11pt]{amsart}
\usepackage[utf8]{inputenc}
\usepackage{geometry,amsmath,amsthm, amssymb,graphicx,tikz,tikz-cd,mathrsfs,fancyhdr,wrapfig,enumitem,mathtools,verbatim}
\usepackage{style}
\usepackage{hyperref}
\usepackage[all]{xy}

\usepackage[maxbibnames=99, style=alphabetic]{biblatex}
\title{Two Formulas for $F$-polynomials}
\author[Feiyang Lin]{Feiyang Lin$^*$}

\thanks{$*$~Department of Mathematics, University of California, Berkeley, CA 94720, USA. fylin@berkeley.edu}

\author[Gregg Musiker]{Gregg Musiker$^\dagger$}

\thanks{$\dagger$~ School of Mathematics, University of Minnesota, 206 Church St. SE, Minneapolis, MN, 55455 USA. musiker@math.umn.edu}

\author[Tomoki Nakanishi]{Tomoki Nakanishi$^\ddagger$}

\thanks{$\ddagger$~ Graduate School of Mathematics, Nagoya University, Furo-cho, Chikusa-ku, Nagoya, 464-8602, Japan. nakanisi@math.nagoya-u.ac.jp}

\begin{document}

\begin{abstract}
We discuss a product formula for $F$-polynomials in cluster algebras, and provide two proofs. One proof is inductive and uses only the mutation rule for $F$-polynomials. The other is based on the Fock-Goncharov decomposition of mutations. We conclude by expanding this product formula as a sum and illustrate applications.  This expansion provides an explicit combinatorial computation of $F$-polynomials in a given seed that depends only on the $\bc$-vectors and $\bg$-vectors along a finite sequence of mutations from the initial seed to the given seed.
\end{abstract}

\maketitle

\section{Introduction}

Cluster algebras are certain commutative algebras with a rich combinatorial
structure, first introduced in \cite{fomin2002cluster}. Their inception
was motivated by the study of semicanonical bases of Lie algebras; but
since then, researchers have made deep connections between cluster algebras
and many other areas of math and physics, including discrete dynamical
systems, Poisson geometry, higher Teichmüller spaces, commutative and
non-commutative algebraic geometry, string theory, and quiver representation
theory; see \cite{Keller2012} for specific references.

The generators of a cluster algebra, called cluster variables, can be naturally described using two pieces of data which were introduced in \cite{FominZelevinsky2007IV}, namely \emph{$F$-polynomials} and \emph{$\bg$-vectors}.  More detail is given in Section \ref{sec:background}. Thus, $F$-polynomials play a key part in the study and applications of cluster patterns. In this paper, we draw attention to a less-known product formula for $F$-polynomials applicable to all skew-symmetrizable cluster algebras (Theorem \ref{thm:main}) that we refer to as \emph{Gupta's formula}. Gupta's formula can also be expanded using the generalized binomial theorem into a summation formula indexed over integer tuples, giving an explicit formula for the coefficients of $F$-polynomials (Theorem \ref{thm:Gupta'sFormulaV2}). 

We begin this paper with some basics on cluster algebras, $F$-polynomials, $\bc$-vectors, and $\bg$-vectors in Sections \ref{sec:background1} and \ref{sec:background2}.  We then recall for the reader the Fock-Goncharov decomposition of mutations \cite{FockGoncharov09} in Section \ref{sec:FG-decomp} and reformulate mutations in terms of two automorphisms of the field of rational functions in the initial cluster variables in Section \ref{sec:Nontropical}.  The main theorem follows in Section \ref{sec:main} with two proofs in Sections \ref{sec:proof1} and \ref{sec:proof2}. These results hold for an arbitrary initial skew-symmetrizable matrix $B_{t_0}$. Finally, we conclude with a second manifestation of our main theorem in Section \ref{sec:Expansion}, this time as a formula involving an infinite sum, rather than as a product formula.

Let us give a brief account of the short history of Gupta's formula, justifying its name.
As part of her REU project \cite{Gupta2018} advised by the second author, Meghal Gupta  discovered both the product and summation versions of this $F$-polynomial formula, and gave an elementary proof for it, up to the use of sign-coherence.  She then used these two formulas to obtain explicit formulas for $F$-polynomials for cluster patterns associated to the $r$-Kronecker quiver, the affine $\cA_{n,1}$ quiver and Gale Robinson quivers, and resolved Eager and Franco's stabilization conjecture \cite[Sec. 9.5]{eager2012colored} for several special cases. 
Comparing with \cite{Gupta2018}, the reader will notice that our current formulation of Gupta’s formula is very different from the original formulation.  For example, instead of her usage of $a_{i,j}$, $b_{i,j}$, and $W(n,w_1,w_2,\dots,w_k)$ in \cite[Def. 2.15 and Def. 2.16]{Gupta2018}, we use dot products and sums involving more familiar cluster algebraic objects such as $\bc$-vectors and $\bg$-vectors.
Gupta's proof for the product version of her formula (Theorem \ref{thm:main}) is embedded in her proof towards the summation version (Theorem \ref{thm:Gupta'sFormulaV2}), but it proceeds by induction in essentially the same manner as the first proof we give, utilizing the same recurrence relation for $F$-polynomials as we do. In particular, Equation (2) from \cite[Sec. 3.2]{Gupta2018} translates to Theorem \ref{thm:Gupta'sFormulaV2} herein. Our second proof is dissimilar to her approach, relying on the Fock-Goncharov decomposition of mutations into a tropical and a non-tropical part.

\subsection*{Funding} This work was supported in part by the Japan Society for the Promotion of Science [16H03922];
and the National Science Foundation Research Training Group [DMS-174563].

\subsection*{Acknowledgements} The authors thank Bernhard Keller and Man-Wai Cheung for useful comments. It is inspired by Meghal Gupta's work in the 2018 REU at University of Minnesota in Combinatorics and Algebra, as well as the first author's work during the 2020 REU as well as her work supported by Harvey Mudd College.  We also thank the anonymous referees who provided numerous suggestions that have improved the exposition of our paper.

\section{Background}
\label{sec:background}
Before delving into the statement of Gupta's formula and its proofs, we will establish essential notation and provide some background on the Fock-Goncharov decomposition of mutations and related automorphisms, closely following \cite{Nakanishi2021}. We assume that the reader is familiar with the basics of cluster algebras, such as the definitions of $F$-polynomials, $\bc$-vectors, $\bg$-vectors in \cite{FominZelevinsky2007IV}, but begin by reviewing some of the salient points.

\subsection{The basics (Part 1)}
\label{sec:background1}
Let $\T_n$ be the $n$-regular tree graph where the $n$ edges attached to each vertex are distinctly labeled by $1, \dots, n$. 
We say that a pair of vertices $t$ and $t'$ in $\T_n$ are \textit{$k$-adjacent} if they are connected with an edge labeled by $k$. 

Let $\bbP=\mathrm{Trop}(\bfy)$ be the tropical semifield generated by $y_1, \dots, y_n$. The tropical semifield comes equipped with two binary operations, namely \emph{tropical addition} which is denoted as $\oplus$, and ordinary multiplication of Laurent monomials.  The tropical addition of two Laurent monomials in $\mathbb{P}$ is defined as 
\[y_1^{d_1}y_2^{d_2}\cdots y_n^{d_n} \oplus
y_1^{e_1}y_2^{e_2}\cdots y_n^{e_n} = 
y_1^{\min(d_1,e_1)}y_2^{\min(d_2,e_2)}\cdots y_n^{\min(d_n,e_n)}\]
for all integers $d_i$ and $e_i$ for $1 \leq i \leq n$.
It is straightforward to show that these two operations turn $\mathbb{P}$ into a semifield, meaning that the commutative, associative, and distributive laws are satisfied. Let $\bbZ\bbP$ be the group ring of the multiplicative group of $\bbP$ over $\bbZ$,
and let $\bbQ\bbP$ be its fraction field. 
In this paper, we consider cluster algebras of geometric type, meaning that their coefficients lie in $\bbP$. 

A cluster algebra (of rank $n$) is defined by an initial \emph{seed} with principal coefficients, which is a tuple $\Sigma_{t_0} = (\x_{t_0} , \y_{t_0} , B_{t_0})$ where $t_0$ is some vertex in $\T_n$, $\x_{t_0} = (x_1, \dots, x_n)$ is an $n$-tuple of algebraically independent elements in the rational field $\mathbb{QP}(\x_{t_0})$, $\y_{t_0} = (y_1, \dots, y_n)$ is the $n$-tuple of tropical generators of $\mathbb{P}$, and $B_{t_0}$ is the initial exchange matrix, which is an $n$-by-$n$ skew-symmetrizable integer matrix, meaning that there exists a diagonal matrix $D$ with positive entries such that $D B_{t_0} = -B_{t_0}^T D$. The \emph{principal extension} of the initial exchange matrix $B_{t_0}$ is denoted as $\widetilde{B_{t_0}}$ and it is the $2n$-by-$n$ matrix whose top $n$ rows is $B_{t_0}$ and whose bottom $n$ rows is an $n$-by-$n$ identity matrix. 

From the initial seed $\Sigma_{t_0}$, we obtain seeds $\Sigma_t$ for all $t \in \T_n$ by an iterative process known as \emph{mutation}: given two $k$-adjacent vertices $t, t' \in \T_n$, and a seed $\Sigma_t = (\x_t , \y_t , B_t)$ with extended exchange matrix $\widetilde{B_t}$, by \textit{mutating $\Sigma_t$ in direction $k$}, we obtain a seed $\Sigma_{t'} = (\x_{t'} , \y_{t'} , B_{t'})$ and an extended exchange matrix $\widetilde{B_{t'}}$. 
This assignment of a seed to each $t \in \T_n$ is a \emph{cluster pattern}, written $\bfSigma = \{\Sigma_t = (\x_t , \y_t , B_t)\}_{t \in \T_n}$. 
In this process, we also obtain two more families of matrices $C_t$ and $G_t$, known as $C$-, $G$-matrices, as well as a collection of polynomials in $\mathbb{Z}[y_1,\dots, y_n]$, known as $F$-polynomials. Write $A = (\mathbf{a}_i)_{i = 1}^n$ to mean that $\mathbf{a}_i$ is the $i$-th column of the $n$-by-$n$ matrix $A$. Then given a seed $\Sigma_t = (\x_t,\y_t,B_t) \in \bfSigma$, we write
\begin{align*}
\x_t &= (x_{1;t},\dots,x_{n;t}), \\
\y_t &= (y_{1;t},\dots,y_{n;t}), \\ 
B_t &= (\bb_{i; t})_{i = 1}^n = (b_{ij;t})_{i,j=1}^{n}, \\
C_t &= (\bc_{i; t})_{i = 1}^n = (c_{ij;t})_{i,j=1}^{n}, \\
G_t &= (\bg_{i; t})_{i = 1}^n = (g_{ij;t})_{i,j=1}^{n}.
\end{align*}

For the initial seed $\Sigma_{t_0} = (\x_{t_0} , \y_{t_0} , B_{t_0})$, we often drop the index $t_0$:
\begin{align*}
\x_{t_0} &= \x = (x_{1},\dots,x_{n}), \\
\y_{t_0} &= \y = (y_{1},\dots,y_{n}), \\
B_{t_0} &= B = (b_{ij})_{i,j=1}^n.
\end{align*}

Let us say a bit more about how $\x_t$, $\y_t$, $B_t$, $C_t$, $G_t$ and $F$-polynomials are obtained via mutations. Let $\mu_{k;t}$ denote the mutation map applied to the seed $\Sigma_t$ in direction $k$. To mutate the exchange matrix $B_t = (b_{ij; t})$, we write $[n]_+ = \max(n,0)$ and define $B_{t'} =
(b_{ij; t'})$ as follows: 
$$b_{ij;t'} = 
\begin{cases} -b_{ij; t} & \mathrm{~if~}i=k \mathrm{~or~} j = k \\
b_{ij;t} + b_{ik;t}[b_{kj;t}]_+ + [-b_{ik;t}]_+ b_{kj;t} & \mathrm{~otherwise.} 
\end{cases}$$

\noindent By applying mutation to $\widetilde{B_{t}}$, we obtain $\widetilde{B_{t'}}$ whose the top $n$ rows yield $B_{t'}$, and the bottom $n$ rows yield $C_{t'}$, the $C$-matrix associated to $\Sigma_{t'}$. Like our initial choice of $B_{t_0}$, each exchange matrix $B_t$ is skew-symmetrizable. We may take a common skew-symmetrizer $D = (d_1^{-1}, \dots, d_n^{-1})$, meaning that $D B_t = -B_t^T D$ for all $t \in \T_n$, and such that $d_1, \dots, d_n$ are positive integers. We fix such a $D$ from now on.
Note that the integers $d_i$’s match the ones in \cite{GrossHackingKeelKontsevich2018} in the scattering diagram formalism.

The mutation of $\y_t = (y_{1;t},\dots , y_{n;t})$, as elements of $\mathbb{P}$, is defined as
$$y_{i; t'} = \begin{cases}
y_{k;t}^{-1} & \mathrm{~if~}i=k \\
y_{i;t} y_{k;t}^{[b_{ki;t}]_+} (1 \oplus y_{k;t})^{-b_{ki;t}} & \mathrm{~if~}i \not = k. 
\end{cases}$$

A related set of auxiliary variables, the $\haty$-variables are defined by
\begin{align*}
\haty_{i;t}=y_{i;t}\prod_{j=1}^n x_{j;t}^{b_{ji;t}}.
\end{align*}
Again for the initial seed, we often drop the index $t_0$ and write
\[
\haty_{i;t_0} = \haty_{i} = y_{i}\prod_{j=1}^n x_{j}^{b_{ji}}.
\]

The effect of $\mu_{k;t}$ on $\x_t$ can be understood as an isomorphism of fields of rational functions. If we fix a choice of $\varepsilon \in \{-1,1\}$, 
then the map $\mu_{k;t}$ is defined as follows:
\begin{align}
\nonumber \mu_{k;t}: \Q\bP(\x_{t'}) \quad & \to \quad \Q\bP(\x_t) \\ 
\label{eq:mux0} x_{i;t'} \quad & \mapsto \quad
\begin{cases}
\frac{x_{k;t}^{-1} 
\left (\prod_{j=1}^n x_{j;t}^{[-\varepsilon b_{jk;t}]_+} \right )~ \left(1~+~\hy_{k;t}^{\varepsilon}\right)} {1~\oplus~ y_{k;t}^{\varepsilon}} & i = k, \\
x_{i;t} & i \neq k.
\end{cases}
\end{align}
It is easy to see that the definition is independent of the choice of $\varepsilon$. This map allows us to understand each $x_{i;t'}$ as an element of $\Q\bP(\x_{t})$. Repeated mutations give rise to a map $\Q\bP(\x_t) \to \Q\bP(\x_{t_0})$, which allows us to understand each $x_{i;t}$ as an element of $\Q\bP(\x_{t_0})$. Since $\mathbb{P} = \mathrm{Trop}(y_1,\dots,y_n)$, we may also view $x_{i;t}$ as a rational function (in fact a Laurent polynomial) in $\mathbb{Q}(\x_{t_0},\y_{t_0})$ such that only monomials in $\x_{t_0}$ appear in the denominator.  
Note that when iterating mutations, the order of the composition of the maps is opposite to the order of the mutations due to the domain and image of the map in \eqref{eq:mux0}.

We define the associated $F$-polynomials $F_{i;t} \in \mathbb{Z}[\y_{t_0}]$ as the specialization of $x_{i;t}|_{\x_{t_0}=1}$. Based on equation \eqref{eq:mux0}, i.e. see Prop. 5.1 of \cite{FominZelevinsky2007IV}, the $F$-polynomials satisfy the recurrence
\begin{align}
\label{eq:FPol}
F_{k;t'}&= F_{k;t}^{-1} \left( \frac{y_{k;t}}{1\oplus y_{k;t}}\prod_{j=1}^n F_{j;t}^{[b_{jk;t}]_+}
        + \frac{1}{1\oplus y_{k;t}}\prod_{j=1}^n F_{j;t}^{[-b_{jk;t}]_+}\right).
\end{align}

\noindent Furthermore, we obtain the $g$-vector $\bfg_{i;t}$ as the exponent vectors of the Laurent monomial arising from the specialization $x_{i;t}|_{\y_{t_0}=0}$.  
The $G$-matrix $G_t$ is defined by concatenating the $\bg$-vectors $\bfg_{1;t}$ through $\bfg_{n;t}$  together so that they form its columns.

\subsection{The basics (Part 2)}
\label{sec:background2}
We remind the reader of some well-known properties that will be useful later in our proofs.

\begin{proposition}[\cite{FominZelevinsky2007IV}, Prop. 6.6] \label{prop:gVectorRecurrence} Let $t', t \in \T_n$ be $k$-adjacent. Then $\bg_{i;t'} = \bg_{i;t}$ if $i \neq k$ and 
\begin{align*}
\bg_{k;t'} &= -\bg_{k;t} - \sum_{j=1}^{n} [c_{jk; t}]_+\bbb_{j;t_0} + \sum_{j=1}^{n}[b_{jk; t}]_+\bg_{j;t} 
\\ &= -\bg_{k;t} - \sum_{j=1}^{n}[-c_{jk;t}]_+ \bbb_{j;t_0} + \sum_{j=1}^{n}[-b_{jk; t}]_+\bg_{j;t}.
\end{align*}
\end{proposition}

\begin{theorem} \label{thm:dualitiesSC} The following holds for any $t \in \T_n$.

(First duality. \cite{FominZelevinsky2007IV}, Eq. (6.14))
\begin{equation*} 
G_tB_t = B_{t_0} C_t;
\end{equation*}
(Second duality. \cite{NakanishiZelevinsky2012}, Eq. (3.11).)
\begin{equation*}
D^{-1}G^T_t D C_t = I;
\end{equation*}
(Sign-coherence. \cite{GrossHackingKeelKontsevich2018}, Cor. 5.5) for all $1 \leq k \leq n$, either $c_{ik; t} \geq 0$ for all $1 \leq i \leq n$, or $c_{ik; t} \leq 0$ for all $1 \leq i \leq n$. In other words, there exists $\varepsilon_{k; t} \in \{-1, 1\}$ such that $\bc_{k; t}^+ = \varepsilon_{k; t} \bc_{k;t}$ consists of only non-negative entries.
\end{theorem} 

The first duality motivates the following definition.
\begin{definition} \label{def:chat} Let $\hat{C}_t = G_tB_t = B_{t_0} C_t$, and let $\hat{C}_t = (\hat{\bc}_{j;t})_{j = 1}^n$. Let $\hat{\bc}_{j;t}^+ = \varepsilon_{j;t}\hat{\bc}_{j;t}$.
\end{definition} Contrary to the case of the $\bc^+$-vectors, as we will see in Example \ref{ex:B14}, we do not expect the entries of $\hat{\bc}_j^+$ to be all non-negative.

Additionally, the presence of the skew-symmetrizer $D$ in the second duality motivates us to consider the inner product $(\bu, \bv)_D = \bu^T D \bv$. We observe that $(\bu, \bv)_D = (\bv, \bu)_D$, and 
\begin{equation} \label{eq:Bantisym}
(\bu, B_t\bv)_D = \bu^T D B_t\bv = -\bu^T B_t^T D\bv = -(B_t \bu, \bv)_D
\end{equation}
for all $t \in \T_n$.

Using this inner product, the second duality can be equivalently stated as follows in terms of individual $\bg$-vectors and $\bc$-vectors: for $1 \leq i, j \leq n$ and any $t \in \T_n$,
\begin{align}\label{eq:secondDualityVecForm}
(\bg_{i; t}, d_j \bc_{j; t})_D = \delta_{i, j}.
\end{align}

Sign-Coherence allows us to rewrite Proposition \ref{prop:gVectorRecurrence} in a simpler form.
\begin{corollary}
\label{cor:gVectorRecurrence} Let $t', t \in \T_n$ be $k$-adjacent. Then $\bg_{i;t'} = \bg_{i;t}$ if $i \neq k$ and 
\begin{align*}
\bg_{k;t'} = -\bg_{k;t}  + \sum_{j=1}^{n}[-\epsilon_{k;t} b_{jk; t}]_+\bg_{j;t}.
\end{align*}
\end{corollary}

\subsection{Fock-Goncharov decomposition for principal coeffcients}
\label{sec:FG-decomp}

Following \cite{Nakanishi2021}, which builds off \cite{FockGoncharov09} in the $\varepsilon=1$ case,
we consider the following decomposition of the map $\mu_{k;t}$:
\begin{align}
\label{2eq:decom1}
\mu_{k;t}=\rho_{k;t}\circ \tau_{k;t},
\end{align}
where the map $\tau_{k;t}$ is the following isomorphism,
\begin{align*}
&\tau_{k;t}:\bbQ\bbP(\bfx_{t'}) \rightarrow \bbQ\bbP(\bfx_t),
\\
&\tau_{k;t}(x_{i;t'})=
\begin{cases}
\displaystyle
x_{k;t}^{-1}\prod_{j=1}^n x_{j;t}^{[-\varepsilon_{k;t} b_{jk;t}]_+}
& i=k,
\\
x_{i;t}
&i\neq k,
\end{cases}
\end{align*}
while the map $\rho_{k;t}$ is the following automorphism,
\begin{align*}
&\rho_{k;t}: \bbQ\bbP(\bfx_{t}) \rightarrow \bbQ\bbP(\bfx_t),
\\
&\rho_{k;t}(x_{i;t})=
x_{i;t}(1+\hat{y}_{k;t}^{\varepsilon_{k;t}})^{-\delta_{i,k}}.
\end{align*}

We call the decomposition \eqref{2eq:decom1} the {\em Fock-Goncharov decomposition\/}
of a mutation ${\mu}_{k;t}$ ({\em with tropical sign}) with respect to the initial vertex $t_0$.
We also call $\tau_{k;t}$ and $\rho_{k;t}$ the {\em tropical part\/}
and the {\em nontropical part\/} of ${\mu}_{k;t}$, respectively. See Section 4 of \cite{Nakanishi2021} for more information.  Comparing this decomposition with our definition of $\mu_{k;t}$ in \eqref{eq:mux0}, we note that here we are choosing $\epsilon$ to be the tropical sign $\epsilon_{k;t}$ and hence the denominator $1\oplus y_{k;t}^{\varepsilon}=1$.

The involution property of the mutations of cluster variables $\x_t$ and $\bg$-vectors $\bg_{i;t}$ is restated as follows.
\begin{prop}
\label{2prop:inv1}
Let $t,t'\in \bbT_n$ be vertices which are $k$-adjacent.
Then, the following relations hold.
\begin{align*}
\mu_{k;t'}\circ \mu_{k;t}&=\mathrm{id},
\\
\tau_{k;t'}\circ \tau_{k;t}&=\mathrm{id}.
\end{align*}
\end{prop}
\begin{proof}
Here, the first equality 
is nothing but the involution of the mutations of $x$-variables.
The second equality 
simply relies on the two identities 
$\varepsilon_{k;t'}=-\varepsilon_{k;t}
\mathrm{~and~}
b_{jk;t'}=-b_{jk;t}$.
\end{proof}

For any $t\in \bbT_n$,
consider a sequence of vertices $t_0$, $t_1$, \dots, $t_{r+1}=t\in \bbT_n$
such that
they  are sequentially adjacent with edges labeled by $k_0$, \dots, $k_{r}$.
Then, we define isomorphisms
\begin{align*}
\mu_{t}^{t_0}&:=\mu_{k_0;t_0}\circ\mu_{k_1;t_1}\circ\dots \circ \mu_{k_{r};t_{r}}
:\bbQ\bbP(\bfx_{t}) \rightarrow \bbQ\bbP(\bfx),
\\
\tau_{t}^{t_0}&:=\tau_{k_0;t_0}\circ\tau_{k_1;t_1}\circ\dots \circ \tau_{k_{r};t_{r}}
:\bbQ\bbP(\bfx_{t}) \rightarrow \bbQ\bbP(\bfx),
\end{align*}
where we apply the automorphsims from right-to-left, and we have set $\bfx = \bfx_{t_0}$.
Thanks to Proposition \ref{2prop:inv1}, $\mu_{t}^{t_0}$ and
$\tau_{t}^{t_0}$ depend only on $t_0$ and $t$.
Namely, we do not have to care about any instances of the redundancy $k_{s+1}=k_s$
in the sequence $k_0,\dots,k_r$.

The following proposition  tells
that the tropical parts $\tau_{t}^{t_0}$ are nothing but the
mutations of the tropical parts of cluster variables ($\bg$-vectors).

\begin{prop}
\label{2prop:trop1}
The following formulas hold:
\begin{align}
\label{2eq:mux1}
\mu_{t}^{t_0}(x_{i;t})&=\x^{\bfg_{i;t}}F_{i;t}(\hat\bfy),\\
\label{2eq:tg1}
\tau_{t}^{t_0}(x_{i;t})&=\x^{\bfg_{i;t}},
\\
\label{2eq:tc1}
\tau_{t}^{t_0}(\haty_{i;t})&=\hat \bfy^{\bfc_{i;t}}.
\end{align}
\end{prop}
\begin{proof}
Equation \eqref{2eq:mux1} is the separation formula, i.e. Proposition 3.13 and Corollary 6.3 of \cite{FominZelevinsky2007IV}.
Note that the denominator of the separation formula is $1$ due to Prop 5.2 of \cite{FominZelevinsky2007IV} since we assume that our cluster pattern $\mathbf{\Sigma}$ has principal coefficients at the initial seed $t_0$.  We next note that \eqref{2eq:tg1} follows from Corollary \ref{cor:gVectorRecurrence}. Finally, \eqref{2eq:tc1} follows from \eqref{2eq:tg1} and Theorem \ref{thm:dualitiesSC}. See Proposition 4.3 of \cite{Nakanishi2021} for more details.
\end{proof}

\subsection{Nontropical parts and $F$-polynomials}
\label{sec:Nontropical}
Observing Proposition \ref{2prop:trop1}, it is clear that the nontropical parts $\rho_{k;t}$ are responsible for
generating and
mutating $F$-polynomials.
Let us make this statement more manifest, following \cite{Nakanishi2021}.

Let us introduce the following automorphisms $\frakq_{k;t}$ of $\bbQ\bbP(\bfx)$ for the initial $x$-variables $\bfx$.
\begin{align*}
\frakq_{k;t}:\bbQ\bbP(\bfx) &\rightarrow \bbQ\bbP(\bfx),
\\
\x^{\bfm} & \mapsto
\x^{\bfm}(1+
\hat{\y}^{\bfc_{k;t}^+}
)^{-(\bfm,d_k \bfc_{k;t})_D}, \text{ $\bfm \in \bbZ^n$}
,
\end{align*}
where $\haty_i$ are the initial $\haty$-variables. Recall that $B = B_{t_0}$. For $\bfn \in \Z^n$, we have
\begin{align*}
\frakq_{k;t}(
\hat \y^{\bfn}
)
= \ & \frakq_{k;t}(
\y^{\bfn}\x^{B\bfn}
) \\ 
= \ & \y^{\bfn}\x^{B\bfn}
(1+\hat \y^{{\bfc}_{k;t}^+}
)^{-(B\bfn,d_k \bfc_{k;t})_D} \\ 
= \ & \hat \y^{\bfn}
(1+\hat \y^{{\bfc}_{k;t}^+}
)^{(\bfn,d_k B\bfc_{k;t})_D},
\end{align*}
where we have used \eqref{eq:Bantisym} at the last equality. But $\hat \bfc_{k;t} = B\bfc_{k;t}$, so
\begin{align}
\label{eq:qAutOnYhat}
\frakq_{k;t}(
\hat \y^{\bfn}
)
&=\hat \y^{\bfn}
(1+\hat \y^{{\bfc}_{k;t}^+}
)^{(\bfn,d_k \hat\bfc_{k;t})_D}.
\end{align}
 
One can easily confirm the following properties.
\begin{prop}[\cite{Nakanishi2021}, Prop. 4.5]
\label{2prop:taurho1}
The following facts hold:
\par
(a). We have the formula
\begin{align*}
&\frakq_{k;t}(\x^{\bfg_{i;t}})=
\x^{\bfg_{i;t}}(1+
\hat{\y}^{\bfc_{k;t}^+}
)^{-\delta_{i,k}}.
\end{align*}
\par
(b). The following relation holds:
\begin{align*}
\tau^{t_0}_t \circ \rho_{k;t}
= \frakq_{k;t}\circ \tau^{t_0}_t.
\end{align*}
\par
(c). If $t'$ and $t$ are $k$-adjacent, we have
\begin{align*}
\frakq_{k;t'}
\circ
\frakq_{k;t}
=\mathrm{id}.
\end{align*}
\end{prop}

Similar to our definitions of $\mu_{t}^{t_0}$ and $\tau_{t}^{t_0}$, given any $t \in \T_n$, we consider a sequence of vertices $t_0$, $t_1$, \dots, $t_{r+1}=t\in \bbT_n$, which are sequentially adjacent by edges labelled by $k_0, \dots, k_r$. Then we define the automorphism
\begin{align*}
\frakq_{t}^{t_0}&:=
\frakq_{k_{0};t_{0}}
\circ\frakq_{k_1;t_1}\circ
\dots
\circ
\frakq_{k_{r};t_{r}}
:
\bbQ\bbP(\bfx) \rightarrow \bbQ\bbP(\bfx).
\end{align*}
Again, by Proposition \ref{2prop:taurho1} (c),
it depends only on $t_0$ and $t$.

One can  separate the tropical and the nontropical
parts of  $\mu^{t_0}_t $
as follows:
\begin{prop}[\cite{Nakanishi2021}, Prop. 4.6]
\label{2prop:sep3}
The following decomposition holds:
\begin{align*}
\mu^{t_0}_t 
= \frakq_t^{t_0}\circ \tau^{t_0}_t
:
\bbQ\bbP(\bfx_t) \rightarrow \bbQ\bbP(\bfx).
\end{align*}
\end{prop}
\begin{proof}
This can be proved by the successive application of
Proposition \ref{2prop:taurho1} (b).
For simplicity, we omit the symbol $\circ$ for the composition of maps below.
First, we note that
\begin{align*}
&\rho_{k_0;t_0}(x_{i})=
x_{i}(1+\hat{y}_{k_0})^{-\delta_{i,k_0}}
=
\frakq_{k_0;t_0}(x_{i}).
\end{align*}
Then, by Proposition \ref{2prop:taurho1} we have (since $\rho_{k_0,t_0}=\frakq_{k_0;t_0}, ~~ \tau_{t_1}^{t_0} = \tau_{k_0;t_0}, ~~ \tau_{t_2}^{t_0} = \tau_{k_0;t_0}\tau_{k_1;t_1}$, etc.)
\begin{align*}
\mu_{t}^{t_0}&=\mu_{k_0;t_0}\mu_{k_1;t_1}\dots  \mu_{k_{r};t_{r}}\\
&=\rho_{k_0;t_0}\tau_{k_0;t_0}\rho_{k_1;t_1}\tau_{k_1;t_1}\dots 
\rho_{k_{r};t_{r}}\tau_{k_{r};t_{r}}\\
&=
\frakq_{k_0;t_0}
\tau_{t_1}^{t_0}\rho_{k_1;t_1}\tau_{k_1;t_1}\dots 
\rho_{k_{r};t_{r}}\tau_{k_{r};t_{r}}\\
&=
\frakq_{k_0;t_0}
\frakq_{k_1;t_1}
\tau_{t_1}^{t_0}
\tau_{k_1;t_1}\dots 
\rho_{k_{r};t_{r}}\tau_{k_{r};t_{r}}\\
&=
\frakq_{k_0;t_0}
\frakq_{k_1;t_1}
\tau_{t_2}^{t_0}
\rho_{k_2;t_2}
\dots 
\rho_{k_{r};t_{r}}\tau_{k_{r};t_{r}}\\
&= \cdots\\
&=
\frakq_{k_0;t_0}
\frakq_{k_1;t_1}\dots
\frakq_{k_{r};t_{r}}
\tau_{t}^{t_0}\\
&=
 \frakq_t^{t_0} \tau^{t_0}_t.
\end{align*}
where $t = t_{r+1}$.  We record this argument as a commutative diagram as illustrated below. 
\begin{align*}
\raisebox{70pt}
{
\begin{xy}
(0,0)*+{\bbQ\bbP(\bfx_t)}="aa";
(20,0)*+{\bbQ\bbP(\bfx_{t_{r}})}="ba";
(20,-12)*+{\bbQ\bbP(\bfx_{t_{r}})}="bb";
(40,-12)*+{\bbQ\bbP(\bfx_{t_{r-1}})}="cb";
(40,-24)*+{\bbQ\bbP(\bfx_{t_{2}})}="cc";
(60,-24)*+{\bbQ\bbP(\bfx_{t_{1}})}="dc";
(60,-36)*+{\bbQ\bbP(\bfx_{t_{1}})}="dd";
(80,0)*+{\bbQ\bbP(\bfx_{t_{0}})}="ea";
(80,-12)*+{\bbQ\bbP(\bfx_{t_{0}})}="eb";
(80,-24)*+{\bbQ\bbP(\bfx_{t_{0}})}="ec";
(80,-36)*+{\bbQ\bbP(\bfx_{t_{0}})}="ed";
(80,-48)*+{\bbQ\bbP(\bfx_{t_{0}})}="ee";
(10,3)*+{\text{\small $\tau_{k_r;t_r}$}};
(50,3)*+{\text{\small $\tau^{t_0}_{t_r}$}};
(60,-9)*+{\text{\small $\tau^{t_0}_{t_{r-1}}$}};
(30,-8)*+{\text{\small $\tau_{k_{r-1};t_{r-1}}$}};
(50,-21)*+{\text{\small $\tau_{k_{1};t_{1}}$}};
(70,-21)*+{\text{\small $\tau_{k_{0};t_{0}}$}};
(70,-33)*+{\text{\small $\tau_{k_{0};t_{0}}$}};
(4,-6)*+{\text{\small $\mu_{k_r;t_r}$}};
(44,-30)*+{\text{\small $\mu_{k_{1};t_{1}}$}};
(64,-42)*+{\text{\small $\mu_{k_0;t_0}$}};
(25,-5)*+{\text{\small $\rho_{k_r;t_r}$}};
(65,-29)*+{\text{\small $\rho_{k_{1};t_{1}}$}};
(86,-6)*+{\text{\small $\frakq_{k_r;t_r}$}};
(86,-30)*+{\text{\small $\frakq_{k_{1};t_{1}}$}};
(92,-42)*+{\text{\small $\rho_{k_0;t_0}=\frakq_{k_0;t_0}$}};
(80,-17)*+{\vdots};
(40,-17)*+{\vdots};
(30,-17)*+{\ddots};
{\ar "aa";"ba"};
{\ar "ba";"bb"};
{\ar "aa";"bb"};
{\ar "ba";"ea"};
{\ar "bb";"cb"};
{\ar "cb";"eb"};
{\ar "cc";"dc"};
{\ar "cc";"dd"};
{\ar "ea";"eb"};
{\ar "dc";"dd"};
{\ar "dc";"ec"};
{\ar "ec";"ed"};
{\ar "dd";"ed"};
{\ar "ed";"ee"};
{\ar "dd";"ee"};
\end{xy}
}
\end{align*}
\end{proof}

We conclude that the automorphisms $\frakq_t^{t_0}$ generate the nontropical parts of cluster variables (namely, $F$-polynomials) in the following manner.
 
\begin{thm}[\cite{Nakanishi2021}, Thm. 4.7] \label{thm:qaut}
 The following formula holds:
  \begin{align} \label{eq:qaut} 
 \frakq_t^{t_0}(\x^{\bfg_{i;t}})
 &=
 \x^{\bfg_{i;t}}
 F_{i;t}(\hat\bfy).
 \end{align}
 \end{thm}
 \begin{proof}
This follows from Propositions
\ref{2prop:trop1} and \ref{2prop:sep3} as
\begin{align*}
 \frakq_t^{t_0}(\x^{\bfg_{i;t}})
=
\frakq_t^{t_0}(\tau_{t}^{t_{0}}(x_{i;t}))
=
\mu_t^{t_0}(x_{i;t})
=
 \x^{\bfg_{i;t}}
 F_{i;t}(\hat\bfy).
\end{align*}
\end{proof}

\section{Statement of Gupta's Formula}
\label{sec:main} 
We are now ready to state the main theorem. 

\begin{theorem}[Gupta's Formula] \label{thm:main}
Given a mutation sequence $\mu_{i_1}\mu_{i_2}\dots$, let $t_j$ be the seed obtained by applying the mutations $\mu_{i_1}\mu_{i_2}\dots\mu_{i_j}$ to the initial seed $t_0$ in the cluster pattern defined by the initial exchange matrix $B_{t_0}$, starting with the application of $\mu_{i_1}$, and proceeding left-to-right.  Here, $B_{t_0}$ is an arbitrary initial skew-symmetrizable matrix.

Let \[ d_{(j)} = d_{i_j}, \quad \bc_{(j)} = \bc_{i_j; t_j}, \quad \bc_{(j)}^+ = \bc_{i_j; t_j}^+, \quad \hat{\bc}_{(j)}^+ = \hat{\bc}_{i_j; t_j}^+, \quad \bg_{(j)} = \bg_{i_j; t_j}, \quad z_j = \hat{\by}^{\bc_{(j)}^+}.\]
Then the $\ell$-th $F$-polynomial along the mutation sequence is 
\[F_{i_\ell; t_\ell}(\hat{\by}) = \prod_{j = 1}^{\ell} L_j^{(\bg_{(\ell)}, d_{(j)} \bc_{(j)})_D}  \text{ where } L_1 = 1+z_1, L_k = 1 + z_k\prod_{j = 1}^{k-1}L_j^{(\hat{\bc}_{(k)}^+, d_{(j)} \bc_{(j)})_D}.\]
\end{theorem}

\begin{remark} The formula can be restated in a different manner using ordinary dot products. We write $C_t^{B; t_0}$ denote the $C$-matrix at seed $t$ in the cluster pattern defined by an initial exchange matrix $B$, when the choice of $B$ is not clear from the context. Let $\nye{C}_t = C^{-B_{t_0}^T; t_0}_{t} = (\tilde{\bc}_{i;t})_{i = 1}^n$, and let $\tilde{\bc}_{(j)} = \tilde{\bc}_{i_j; t_j}$.  Then
\[F_{i_\ell; t_\ell}(\hat{\by}) = \prod_{j = 1}^{\ell} L_j^{\tilde{\bc}_{(j)} ~\cdot~ \bg_{(\ell)}}  \text{ where } L_1 = 1+z_1, L_k = 1 + z_k\prod_{j = 1}^{k-1}L_j^{\tilde{\bc}_{(j)} ~\cdot~ \hat{\bc}_{(k)}^+}.\]
To see that these two formulations agree, we recall the following equality (\cite{NakanishiZelevinsky2012}, Eq. (2.7)): 
\[
C_{t}^{-B^T; t_0} = D C_{t}^{B; t_0} D^{-1}.
\]
It follows that $\tilde{\bc}_{(j)} = D d_{(j)} \bc_{(j)}$. Thus
\begin{align*}
(\bg_{(\ell)}, d_{(j)} \bc_{(j)})_D &= \bg_{(\ell)}^T Dd_{(j)} \bc_{(j)} = \bg_{(\ell)}^T\tilde{\bc}_{(j)} = \tilde{\bc}_{(j)} \cdot \bg_{(\ell)}, \\
(\hat{\bc}_{(k)}^+, d_{(j)} \bc_{(j)})_D 
&= \hat{\bc}_{(k)}^{+T} Dd_{(j)} \bc_{(j)} 
= \hat{\bc}_{(k)}^{+T} \tilde{\bc}_{(j)} = \tilde{\bc}_{(j)} \cdot \hat{\bc}_{(k)}^+.
\end{align*}
\end{remark}

\begin{example} \label{ex:B14}
Let $\mu = \mu_1 \mu_2 \mu_1$ and let 
\[
B_{t_0} = \mattwo{0}{1}{-4}{0}.
\]

Then 
\[
D = \mattwo{1}{0}{0}{\frac{1}{4}},
\]
and $d_{(1)} = d_1 = 1$, $d_{(2)} = d_2 = 4$, and $d_{(3)} = d_1 = 1$.
We will compute the $F$-polynomial $F_{1, t_3}(y_1, y_2)$.
The $\bc, ~~ \hat{\bc}^+, ~~ \bg$-vectors involved are as follows:
\begin{align*}
\bc_{(1)} &= \vectwo{-1}{0}, \quad
\bc_{(2)} = \vectwo{-1}{-1}, \quad
\bc_{(3)} = \vectwo{-3}{-4}, \\  
\hat{\bc}_{(1)}^+ &= \vectwo{0}{-4}, \quad
\hat{\bc}_{(2)}^+ = \vectwo{1}{-4}, \quad
\hat{\bc}_{(3)}^+ = \vectwo{4}{-12}, \\ 
\bg_{(1)} &= \vectwo{-1}{4}, \quad
\bg_{(2)} = \vectwo{-1}{3}, \quad
\bg_{(3)} = \vectwo{-3}{8}.
\end{align*}
So $z_1 = \hy_1$, $z_2 = \hy_1\hy_2$, $z_3 = \hy_1^3 \hy_2^4$. The relevant inner products are 
\begin{align*}
& (\bg_{(3)}, d_{(1)}\bc_{(1)})_D = 3, \quad
(\bg_{(3)}, d_{(2)}\bc_{(2)})_D = 4, \quad
(\bg_{(3)}, d_{(3)}\bc_{(3)})_D = 1, \\  
&(\hat{\bc}_{(2)}^+, d_{(1)} \bc_{(1)})_D = -1, \quad
(\hat{\bc}_{(3)}^+, d_{(1)} \bc_{(1)})_D = -4, \quad
(\hat{\bc}_{(3)}^+, d_{(2)} \bc_{(2)})_D = -4.
\end{align*}
Therefore, 
\begin{align*}
L_1 &= 1+z_1 = 1+\hy_1, \\  
L_2 
&= 1+z_2L_1^{(\hat{\bc}_{(2)}^+, d_{(1)} \bc_{(1)})_D} 
= 1+\hy_1\hy_2(1+\hy_1)^{-1}, \\ 
L_3 
&= 1+z_3L_1^{(\hat{\bc}_{(3)}^+, d_{(1)} \bc_{(1)})_D }L_2^{(\hat{\bc}_{(3)}^+, d_{(2)} \bc_{(2)})_D} \\
&= 1+\hy_1^3\hy_2^4(1+\hy_1)^{-4}(1+\hy_1\hy_2(1+\hy_1)^{-1})^{-4} \\ 
&= 1+\hy_1^3\hy_2^4(1+\hy_1+\hy_1\hy_2)^{-4}. 
\end{align*}
Applying Gupta's formula then gives us
\begin{align}
F_{1,t_3}(\hy_1, \hy_2) &= L_1^{(\bg_{(3)}, d_{(1)}\bc_{(1)})_D}L_2^{(\bg_{(3)}, d_{(2)}\bc_{(2)})_D} L_3^{(\bg_{(3)}, d_{(3)}\bc_{(3)})_D} \nonumber \\ 
&= L_1^{3}L_2^{4}L_3 \nonumber \\
&= (1+\hy_1)^3(1+\hy_1\hy_2(1+\hy_1)^{-1})^4(1+\hy_1^3\hy_2^4(1+\hy_1+\hy_1\hy_2)^{-4}) \nonumber \\
&= \frac{(1+\hy_1+\hy_1\hy_2)^4+\hy_1^3\hy_2^4}{1+\hy_1} \nonumber \\ 
&= 1 + 3 \hy_1 + 3 \hy_1^2 + \hy_1^3 + 4 \hy_1 \hy_2 + 8 \hy_1^2 \hy_2 + 4 \hy_1^3 \hy_2 + 
 6 \hy_1^2 \hy_2^2 + 6 \hy_1^3 \hy_2^2 + 4 \hy_1^3 \hy_2^3 + \hy_1^3 \hy_2^4. \label{eq:F3expanded}
\end{align}
\end{example}

\section{First Proof of Theorem \ref{thm:main}} \label{sec:proof1}

Recall that $[n]_+ = \max(n, 0)$. Given a vector $\bu = (u_1, \dots, u_n)$, let $[\bu]_+ = ([u_1]_+, \dots, [u_n]_+)$.

\begin{proof}
We shall prove the following formula for each $F$-polynomial at the seed $t_\ell$, which specializes to Theorem \ref{thm:main} when $i = i_{\ell}$:
\begin{equation} \label{eq:GuptaForAllFPolynomials}
F_{i; t_\ell} = \prod_{j = 1}^{\ell} L_j^{(\bg_{i; t_\ell}, d_{(j)} \bc_{(j)})_D} \text{ where } L_1 = 1+z_1, L_k = 1 + z_k\prod_{j = 1}^{k-1}L_j^{(\hat{\bc}_{(k)}^+, d_{(j)} \bc_{(j)})_D}.
\end{equation}

We proceed by induction on $\ell$. The base case is $\ell = 0$, where the formula above  reduces to the empty product, which we interpret to be $1$ by convention. Since we are at the initial seed, the $F$-polynomial for each cluster variable is just $1$. So the $F$-polynomials at $t_0$ agree with the formula.

Now suppose that the formula is correct for some $\ell \geq 0$, and let $k = i_{\ell+1}$, $t = t_\ell$ and $t' = t_{\ell+1}$. By \eqref{eq:secondDualityVecForm}, $(\bg_{i;t'}, d_k \bc_{(\ell+1)})_D = \delta_{i,k}$. (Note here that $d_k = d_{(\ell+1)}$.)  So if $i \neq k$, since $\bg_{i; t'} = \bg_{i; t}$, for all $1 \leq j \leq \ell$,
\begin{align*}
(\bg_{i;t'}, d_{(j)} \bc_{(j)})_D &= (\bg_{i;t}, d_{(j)} \bc_{(j)})_D,
\end{align*} 
and 
\[
(\bg_{i;t'}, d_k \bc_{(\ell+1)})_D = 0.
\]
Therefore, when $i \neq k$, 
\[F_{i;t'} = F_{i;t} = \prod_{j = 1}^{\ell} L_j^{(\bg_{i; t}, d_{(j)} \bc_{(j)})_D} = \prod_{j = 1}^{\ell+1} L_j^{(\bg_{i; t'}, d_{(j)} \bc_{(j)})_D}.\]
It remains to prove \eqref{eq:GuptaForAllFPolynomials} for $F_{k; t'}$. By the recurrence of $F$-polynomials, i.e. equation \eqref{eq:FPol}, we know that
\[F_{k; t'} = 
    \frac{
    \hat{\by}^{[\bc_{k;t}]_+} \prod_{i=1}^{n} F_{i;t}^{[b_{ik; t}]_+} + 
    \hat{\by}^{[-\bc_{k;t}]_+} \prod_{i=1}^{n} F_{i;t}^{[-b_{ik; t}]_+}
    }{F_{k; t}}.\]
Assuming that \eqref{eq:GuptaForAllFPolynomials} is true for $t_\ell$, we can substitute $F_{i;t}$ in the numerator and $F_{k;t}$ in the denominator by the appropriate products of the $L_j$'s, obtaining
\[
    F_{k; t'} = \hat{\by}^{[\bc_{k;t}]_+} \prod_{j = 1}^{\ell}L_j^{
    (\sum_{i=1}^{n}[b_{ik; t}]_+ \bg_{i;t} - \bg_{k;t}, \ 
    d_{(j)} \bc_{(j)})_D
    }
    + \hat{\by}^{[-\bc_{k;t}]_+} \prod_{j = 1}^{\ell}L_j^{
    (\sum_{i=1}^{n}[-b_{ik; t}]_+ \bg_{i;t} - \bg_{k;t},
    \ d_{(j)} \bc_{(j)})_D}
    ,
\]
where the $-\bg_{k;t}$ comes from $F_{k;t}$ in the denominator.

By Proposition \ref{prop:gVectorRecurrence}, 
\[\sum_{i=1}^{n}[b_{ik; t}]_+ \bg_{i;t} - \bg_{k;t} = \bg_{k;t'}+\sum_{j=1}^{n}[c_{jk; t}]_+ \bbb_{j;t_0} = \bg_{k;t'}+B_{t_0}[\bc_{k;t}]_+\]
and 
\[\sum_{i=1}^{n}[-b_{ik;t}]_+ \bg_{i;t} - \bg_{k;t} = \bg_{k;t'}+\sum_{j=1}^{n}[-c_{jk; t}]_+ \bbb_{j;t_0} = \bg_{k;t'}+B_{t_0}[-\bc_{k;t}]_+.\]
So 
\[F_{k; t'} =
\hat{\by}^{[\bc_{k;t}]_+} \prod_{j = 1}^{\ell}L_j^{
	(\bg_{k;t'}+B_{t_0}[\bc_{k;t}]_+, \ d_{(j)} \bc_{(j)})_D
    }
    + \hat{\by}^{[-\bc_{k;t}]_+} \prod_{j = 1}^{\ell}L_j^{
    (\bg_{k;t'}+B_{t_0}[-\bc_{k;t}]_+, \ d_{(j)} \bc_{(j)})_D
    }.
    \]
If $\varepsilon_{k;t} = 1$,
\[F_{k; t'} = \hat{\by}^{\bc_{k;t}} \prod_{j = 1}^{\ell}L_j^{
    (\bg_{k;t'}+B_{t_0}\bc_{k;t}, \ d_{(j)} \bc_{(j)})_D
    }
    + \prod_{j = 1}^{\ell}L_j^{
    (\bg_{k;t'}, \ d_{(j)} \bc_{(j)})_D};\]
and if $\varepsilon_{k;t} = -1$,
\[F_{k; t'} = \prod_{j = 1}^{\ell}L_j^{
    (\bg_{k;t'}, \ d_{(j)} \bc_{(j)})_D
    }
    + \hat{\by}^{-\bc_{k;t}} \prod_{j = 1}^{\ell}L_j^{
    (\bg_{k;t'}-B_{t_0}\bc_{k;t}, \ d_{(j)} \bc_{(j)})_D
    }.
    \]
We can combine these two cases:
\begin{align*}
F_{k; t'} &= \prod_{j = 1}^{\ell}L_j^{
    	(\bg_{k;t'}, \ d_{(j)} \bc_{(j)})_D
    }
    (
    	1 + \hat{\by}^{\bc_{k;t}^+} \prod_{j = 1}^{\ell}
    	L_j^{
    	(B_{t_0}\bc_{k;t}^+, d_{(j)} \bc_{(j)})_D}
    ) \\ 
    &= \prod_{j = 1}^{\ell}L_j^{
    (\bg_{k;t'}, \ d_{(j)} \bc_{(j)})_D}
    \left (1 + z_{\ell+1} \prod_{j = 1}^{\ell}
    L_j^{(\hat{\bc}_{k;t}^+, d_{(j)} \bc_{(j)})_D} \right ) \\ 
    &= \left(\prod_{j = 1}^{\ell}L_j^{
    (\bg_{k;t'}, \ d_{(j)} \bc_{(j)})_D}\right) L_{\ell+1}
\end{align*}

Recall that $(\bg_{i;t'}, d_k \bc_{(\ell+1)}) = \delta_{i, k}$, so $(\bg_{k; t'}, d_k \bc_{(\ell+1)})_D = 1$. Hence
\[F_{k; t'} = \prod_{j = 1}^{\ell+1}L_j^{
    (\bg_{k; t'}, d_{(j)} \bc_{(j)})_D}\] as desired.

\end{proof}

\section{Second Proof of Theorem \ref{thm:main}} \label{sec:proof2}
\begin{proof} 
By Theorem \ref{thm:qaut}, it suffices to show that
\[
\fq^{t_0}_{t_\ell}(\x^{\bg_{(\ell)}})
= \x^{\bg_{(\ell)}}
\prod_{j = 1}^{\ell} L_j^{(\bg_{(\ell)}, d_{(j)} \bc_{(j)})_D}.
\]

We first prove that for $0 \leq m \leq k-1$,
\begin{equation} \label{eq:lemmayhat}
\fq^{t_0}_{t_m}(z_k) = z_k\prod_{j=1}^{m}L_j^{(\hat{\bc}_{(k)}^+,d_{(j)} \bc_{(j)})_D}.
\end{equation} 

We prove \eqref{eq:lemmayhat} for a fixed $k$ by induction on $m$. The automorphism $\fq^{t_0}_{t_0}$ is the identity and the empty product is understood to be $1$, so the claim follows when $m = 0$. Now suppose that it is true for some $0 \leq m < k-1$. We wish to prove it for $m+1$. Recall from \eqref{eq:qAutOnYhat} that
\[
\fq_{i;t}(\hat{\by}^{\bn}) = \hat{\by}^{\bn}(1+\hat{\by}^{\bc_{i;t}^+})
^{(\bn, d_i \hat\bc_{i;t})_D}.
\]
Therefore 
\[
\fq_{i_{m+1}, t_{m}}(z_k) = 
z_k(1+\hat\y^{\bc_{i_{m+1}, t_{m}}^+})^{( \bc_{(k)}^+, d_{(m+1)} \hat{\bc}_{i_{m+1}; t_m})_D}.
\]
Since $\bc_{i_{m+1}, t_{m}} = -\bc_{i_{m+1}, t_{m+1}}$, we have $\bc_{i_{m+1}, t_{m}}^+ = \bc_{(m+1)}^+$ and $\hat{\bc}_{i_{m+1}; t_m} = -\hat{\bc}_{(m+1)}.$ So 
\begin{align*}
( \bc_{(k)}^+, d_{(m+1)} \hat{\bc}_{i_{m+1}; t_m})_D &= 
( \bc_{(k)}^+, - d_{(m+1)} B_{t_0} \bc_{(m+1)})_D \\ &= 
(B_{t_0}  \bc_{(k)}^+, d_{(m+1)}  \bc_{(m+1)})_D \\ &= 
(\hat{\bc}_{(k)}^+, d_{(m+1)}  \bc_{(m+1)})_D,
\end{align*}
which allows us to write
\[
\fq_{i_{m+1}, t_{m}}(z_k) =
z_k(1+z_{m+1})^{(\hat{\bc}_{(k)}^+, d_{(m+1)}  \bc_{(m+1)})_D}. 
\]
Using the induction hypothesis and substituting in our calculation of $\fq_{i_{m+1}, t_{m}}(z_k)$,
\begin{align*}
\fq^{t_0}_{t_{m+1}}(z_k) 
&= \fq^{t_0}_{t_{m}}(\fq_{i_{m+1}, t_{m}}(z_k)) \\ 
&= \fq^{t_0}_{t_{m}}
\left (
z_k(1+z_{m+1})^{(\hat{\bc}_{(k)}^+, d_{(m+1)}  \bc_{(m+1)})_D}
\right ) \\ 
&= z_k \prod_{j=1}^{m} L_j^{(\hat{\bc}_{(k)}^+,d_{(j)} \bc_{(j)})_D} 
\left(
	1+z_{m+1}
	\prod_{j=1}^{m}
	L_j^{(\hat{\bc}_{(m+1)}^+,d_{(j)} \bc_{(j)})_D}
\right)^{
	(\hat{\bc}_{{(k)}}^+, d_{(m+1)}\bc_{(m+1)})_D
  } \\ 
&= z_k\prod_{j=1}^{m+1} L_j^{(\hat{\bc}_{(k)}^+,d_{(j)} \bc_{(j)})_D} 
\end{align*}

This concludes the proof of \eqref{eq:lemmayhat}. With it, we compute that for all $1 \leq j \leq \ell$,
\begin{align*}
\frac{
\fq^{t_0}_{t_{j}}(\x^{\bg_{(\ell)}})
}
{
\fq^{t_0}_{t_{j-1}}(\x^{\bg_{(\ell)}})
}
&= \fq^{t_0}_{t_{j-1}}
\left (
\frac{\fq_{i_j, t_{j-1}}(\x^{\bg_{(\ell)}})}{\x^{\bg_{(\ell)}}}
\right ) \\
&= \fq^{t_0}_{t_{j-1}} 
\left (
(1+\hat{\by}^{\bc_{i_{j}, t_{j-1}}^+})^{-(\bg_{(\ell)}, d_{(j)}\bc_{i_j, t_{j-1}})_D} 
\right ) \\
&= \fq^{t_0}_{t_{j-1}} 
\left (
(1+z_j)^{(\bg_{(\ell)}, d_{(j)} \bc_{(j)})_D} 
\right ) \\ 
&= 
(1+\fq^{t_0}_{t_{j-1}} (z_j))^{(\bg_{(\ell)}, d_{(j)} \bc_{(j)})_D} 
\\ 
&= 
\left (
1+z_j\prod_{k=1}^{j-1}L_k^{(\hat{\bc}_{(j)}^+,d_{(k)} \bc_{(k)})_D}\right )^{(\bg_{(\ell)}, d_{(j)} \bc_{(j)})_D}  \\ 
&= L_j^{(\bg_{(\ell)}, d_{(j)} \bc_{(j)})_D}.
\end{align*}

It follows that
\begin{align*}
\frac
{\fq^{t_0}_{t_\ell}(\x^{\bg_{(\ell)}})}
{\x^{\bg_{(\ell)}}} 
&= 
\prod_{j = 1}^{\ell} 
\frac{
\fq^{t_0}_{t_{j}}(\x^{\bg_{(\ell)}})
}
{
\fq^{t_0}_{t_{j-1}}(\x^{\bg_{(\ell)}})
} \\ 
&= \prod_{j = 1}^{\ell} L_j^{(\bg_{(\ell)}, d_{(j)} \bc_{(j)})_D}
\end{align*}

as desired.

\end{proof}

\begin{remark}
\label{rem:KC}
 After Gupta's work \cite{Gupta2018} was posted on the arXiv, Man-Wai Cheung and Bernhard Keller separately reached out to Gupta and the second author, both pointing out that Gupta's formula was related to some known results in the literature.

 Keller pointed out that Gupta's formula can be obtained from Theorem 6.4 of  \cite{Keller2012}, which was first proven in \cite{Nagao2013}.
 Cheung noticed that after translating the relevant notation arising in scattering diagrams, Gupta's formula can be deduced from the formula for $F$-polynomials in terms of the path-ordered products of \cite{GrossHackingKeelKontsevich2018}.

In fact, in the skew-symmetric case,
the formula \eqref{eq:qaut} is  immediately obtained by setting $q=1$
in   Theorem 6.4 of \cite{Keller2012}, which is the quantum version of  the formula \eqref{eq:qaut}.
Also, the automorphism $ \frakq_t^{t_0}$ is identified with a path-ordered product in the corresponding scattering diagram in \cite{GrossHackingKeelKontsevich2018}.
Then,  the formula \eqref{eq:qaut} is immediately obtained from
Theorem 4.9 of \cite{GrossHackingKeelKontsevich2018}.
(See also \cite{reading2020combinatorial}, \cite{Nakanishi2021}.)
Thus, both approaches prove Gupta's formula via the formula \eqref{eq:qaut},
and the  proof presented here, which is based on  the formula \eqref{eq:qaut},
 should coincide with their suggestions,
 while we emphasize that we proved the formula \eqref{eq:qaut} without referring
to the quantum case or Theorem 4.9 of \cite{GrossHackingKeelKontsevich2018}.

Meanwhile, Gupta's formula itself and our proof of the formula \eqref{eq:qaut}
 crucially depend  on the sign-coherence property proved by \cite{GrossHackingKeelKontsevich2018}.
However, in \cite{Nakanishi2021}, it was also clarified that the sign-coherence is proved
without Theorem 4.9 of \cite{GrossHackingKeelKontsevich2018}.
Thus, our proof here together with the proof of the formula \eqref{eq:qaut} is still  in
a different perspective to the above ones.
\end{remark}

\section{Expansion of Gupta's Formula}
\label{sec:Expansion}
One may wish to expand the product formula given in Theorem \ref{thm:main} into a sum. To do so, we first prove a lemma.
\begin{lemma} \label{lem:ex3.2}
Given integers $h_1, \dots, h_\ell$ and the same setup as Theorem \ref{thm:main},
\[\prod_{j = 1}^{\ell} L_j^{h_j}
= \sum_{(m_1, \dots, m_\ell) \in \Z_{\geq 0}} \prod_{j=1}^{\ell}
\binom{h_j + \sum_{k=j+1}^{\ell} m_k (\hat \bc_{(k)}^+, d_{(j)} \bc_{(j)})_D}{m_j}
\hat{\by}^{\sum_{j=1}^{\ell}m_j \bc_{(j)}^+}.
\]
\end{lemma}
\begin{proof}
We prove the following claim by induction: for all $1 \leq p \leq \ell$, 
\begin{equation} \label{eq:indHypothesis}
\prod_{j = p}^{\ell} L_j^{h_j}
= \sum_{(m_p, \dots, m_\ell) \in \Z_{\geq 0}} 
\prod_{j=1}^{p-1} L_j^{\sum_{k=p}^{\ell} m_k (\hat \bc_{(k)}^+, d_{(j)} \bc_{(j)})_D}
\prod_{j = p}^{\ell} \binom{h_j + \sum_{k=j+1}^{\ell} m_k (\hat \bc_{(k)}^+, d_{(j)} \bc_{(j)})_D}{m_j}
\hat{\by}^{\sum_{j=p}^{\ell}m_j \bc_{(j)}^+}.
\end{equation} When $p = 1$, this claim specializes to our lemma.

First note that by the Generalized Binomial Theorem, for any $h \in \Z$, we can write
\begin{align} 
    L_k^{h} &= \left(1 + z_k \prod_{j = 1}^{k-1} L_j^{(\hat \bc_{(k)}^+, d_{(j)} \bc_{(j)})_D} \right)^{h} \nonumber \\ 
    &= \sum_{m \in \Z_{\geq 0}} \binom{h}{m}z_k^{m} \prod_{j = 1}^{k-1} L_j^{m(\hat \bc_{(k)}^+, d_{(j)} \bc_{(j)})_D} \nonumber \\
    &= \sum_{m \in \Z_{\geq 0}} \prod_{j=1}^{k-1} L_j^{m(\hat \bc_{(k)}^+, d_{(j)} \bc_{(j)})_D} \binom{h}{m} \hat{\by}^{m  \bc_{(k)}^+}. \label{eq:LkExpansion}
\end{align}
If we let $k = \ell$ and $h = h_\ell$, the above is precisely the base case $p = \ell$ of our claim.

Now suppose that \eqref{eq:indHypothesis} is true for $p+1$. Multiplying both sides by $L_p^{h_p}$, we get
\begin{align}
\prod_{j = p}^{\ell} L_j^{h_j} 
=
\sum_{(m_{p+1}, \dots, m_\ell) \in \Z_{\geq 0}}
     \ & \bigg ( \ \prod_{j=1}^{p-1} L_j^{
    \sum_{k=p+1}^{\ell} m_k (\hat \bc_{(k)}^+, d_{(j)} \bc_{(j)})_D
    }
    L_p^{
    h_p + \sum_{k=p+1}^{\ell} m_k (\hat \bc_{(k)}^+, d_{(p)}\bc_{(p)})_D
    } \nonumber \\ 
     & \prod_{j = p+1}^{\ell} \binom{
    h_j + \sum_{k=j+1}^{\ell} m_k (\hat \bc_{(k)}^+, d_{(j)} \bc_{(j)})_D
    }{m_j}
    \hat{\by}^{\sum_{j=p+1}^{\ell}m_j \bc_{(j)}^+} \bigg ). \label{eq:indIntermediate}
\end{align}

Letting $k = p$ and $h = h_p + \sum_{k=p+1}^{\ell} m_k (\hat \bc_{(k)}^+, d_{(p)}\bc_{(p)})_D$, our computation in Eq. \eqref{eq:LkExpansion} allows us to expand the middle term of \ref{eq:indIntermediate}:
\[
L_p^{h_p + \sum_{k=p+1}^{\ell} m_k (\hat \bc_{(k)}^+, d_{{(p)}}\bc_{(p)})_D} 
= \sum_{m_p \in \Z_{\geq 0}} \prod_{j=1}^{p-1} L_j^{m_p(\hat\bc_{(p)}^+, d_{(j)} \bc_{(j)})_D} \binom{h_p + \sum_{k=p+1}^{\ell} m_k (\hat \bc_{(k)}^+, d_{{(p)}}\bc_{(p)})_D}{m_p} \hat{\by}^{m_p \bc_{(p)}^+}.
\]
Therefore, Eq. \eqref{eq:indIntermediate} reorganizes into
\[
\prod_{j = p}^{\ell} L_j^{h_j}
= \sum_{(m_p, \dots, m_\ell) \in \Z_{\geq 0}} 
\prod_{j=1}^{p-1} L_j^{\sum_{k=p}^{\ell} m_k (\hat \bc_{(k)}^+, d_{(j)} \bc_{(j)})_D}
\prod_{j = p}^{\ell} \binom{h_j + \sum_{k=j+1}^{\ell} m_k (\hat \bc_{(k)}^+, d_{(j)} \bc_{(j)})_D}{m_j}
\hat{\by}^{\sum_{j=p}^{\ell}m_j \bc_{(j)}^+}
\]
as desired. This completes our proof.
\end{proof}

Applying Lemma \ref{lem:ex3.2} to Theorem \ref{thm:main} yields the following alternative form of Gupta's formula.
\begin{theorem}\label{thm:Gupta'sFormulaV2}
Under the same conditions as Theorem \ref{thm:main}, 
\[F_{i_\ell; t_\ell}(\hat{\by}) = \sum_{(m_1, \dots, m_\ell) \in \Z_{\geq 0}} \prod_{j=1}^{\ell}
\binom{(\bg_{(\ell)}, d_{(j)} \bc_{(j)})_D + \sum_{k=j+1}^{\ell}m_k(\hat  \bc_{(k)}^+, d_{(j)} \bc_{(j)})_D}{m_j}
\hat{\by}^{\sum_{j=1}^{\ell}m_j \bc_{(j)}^+}.\]
\end{theorem}

\begin{remark} 
An expression resembling Theorem \ref{thm:Gupta'sFormulaV2} also appears in Gupta's original formulation, in particular, see Theorem 3.1 of \cite{Gupta2018}.  However one then has to use Definitions 2.16 and 2.17 of \cite{Gupta2018}, as well as translation into the language of $\mathbf{c}$- and $\mathbf{g}$-vectors, as well as algebraic manipulation of factorials, to result in the current expression. 
\end{remark}

\begin{example} We now try to reproduce the result of Example \ref{ex:B14} using Theorem \ref{thm:Gupta'sFormulaV2}. Back in Example \ref{ex:B14}, we have already done the relevant computations involved in the statement of the new formula, giving us
\[
F_{1;t_3}(\hat\by) = \sum_{(m_1, m_2, m_3) \in \Z_{\geq 0}} 
\binom{3 -m_2 - 4m_3}{m_1}
\binom{4 - 4m_3}{m_2}
\binom{1}{m_3}
\hy_1^{m_1 + m_2+3m_3}
\hy_2^{m_2+4m_3}.
\]
Since $\binom{N}{s} = 0$ when $N > 0$ and $s < 0$, or when $N > 0$ and $s > N$, we may restrict our attention to certain tuples. Table \ref{table1} captures ten tuples and their contributions to $F_{1;t_3}(\hat\by)$. For example, the first row considers all the tuples of the form $(m_1, 0, 0)$. Since the first binomial evalutes to $\binom{3}{m_1}$, there is only a nonzero contribution when $m_1 = 0, 1, 2, 3$, yielding four terms that collect into $(1+\hy_1)^3$.
\begin{table}
\centering
\[
\begin{array}{l|l|l}
m_3 & m_2 & \text{terms} \\ \hline
0 & 0 & (1+\hy_1)^3 \\  
  & 1 & 4\hy_1\hy_2(1+\hy_1)^2 \\ 
  & 2 & 6\hy_1^2\hy_2^2(1+\hy_1) \\ 
  & 3 & 4\hy_1^3\hy_2^3
\end{array}
\]
\caption{Tuples and their contributions to $F_{i;t_3}$}
\label{table1}
\end{table}
However, this table is not exhaustive. It leaves two slightly different cases that we will consider separately. When $m_3 = 0$, the second binomial evalutes to $\binom{4}{m_2}$, therefore $m_2$ may also take the value $m_2 = 4$. The contribution of tuples of the form $(m_1, 4, 0)$ is 
\begin{equation}\label{eq:ex40}
\hy_1^4\hy_2^4 \sum_{m_1 \geq 0}\binom{-1}{m_1}\hy_1^{m_1}
= \hy_1^4\hy_2^4 \sum_{m_1 \geq 0}(-1)^{m_1}\hy_1^{m_1}.
\end{equation}
At this point we have considered all tuples where $m_3 = 0$. If $m_3 = 1$, in order for the second binomial to be nonzero, we must have $m_2 = 0$, and the contribution of tuples of the form $(m_1, 0, 1)$ is 
\begin{equation}\label{eq:ex01}
\hy_1^3\hy_2^4 \sum_{m_1 \geq 0}\binom{-1}{m_1}\hy_1^{m_1}
= \hy_1^3\hy_2^4 \sum_{m_1 \geq 0}(-1)^{m_1}\hy_1^{m_1}.
\end{equation}
Summing \eqref{eq:ex40} and \eqref{eq:ex01} leaves a single term $\hy_1^3\hy_2^4$. Since $m_3$ can only take the values $0$ or $1$, we have now finished enumerating all tuples with a nonzero contribution to the formula, allowing us to conclude that
\[
F_{1;t_3}(\hat\by) = (1+\hy_1)^3 + 4\hy_1\hy_2(1+\hy_1)^2 + 6\hy_1^2\hy_2^2(1+\hy_1) + 4\hy_1^3\hy_2^3 + \hy_1^3\hy_2^4,
\]
which agrees with Eq. \eqref{eq:F3expanded}.
\end{example}

\printbibliography

@article{eager2012colored,
  title={Colored BPS pyramid partition functions, quivers and cluster transformations},
  author={Eager, Richard and Franco, Sebasti{\'a}n},
  journal={Journal of High Energy Physics},
  volume={2012},
  number={9},
  pages={1--44},
  year={2012},
  publisher={Springer}
}

@article{FockGoncharov09,
     author = {Fock, Vladimir V. and Goncharov, Alexander B.},
     title = {Cluster ensembles, quantization and the dilogarithm},
     journal = {Annales Scientifiques de l'\'Ecole Normale Sup\'erieure},
     pages = {865--930},
     publisher = {Soci\'et\'e math\'ematique de France},
     volume = {Ser. 4, 42},
     number = {6},
     year = {2009},
    %  doi = {10.24033/asens.2112},
     zbl = {1180.53081},
     mrnumber = {2567745}
    %  language = {en},
    %  url = {http://www.numdam.org/articles/10.24033/asens.2112/}
}

@article{fomin2002cluster,
  title={Cluster algebras I: foundations},
  author={Fomin, Sergey and Zelevinsky, Andrei},
  journal={Journal of the American Mathematical Society},
  volume={15},
  number={2},
  pages={497--529},
  year={2002}
}

@article{FominZelevinsky2007IV,
      title={Cluster algebras {IV}: coefficients},
      author={Fomin, Sergey and Zelevinsky, Andrei},
      journal={Compositio Mathematica},
      volume={143},
      number={1},
      pages={112--164},
      year={2007},
      publisher={London Mathematical Society}
}

@article{GrossHackingKeelKontsevich2018,
      title={Canonical bases for cluster algebras},
      author={Gross, Mark and Hacking, Paul and Keel, Sean and Kontsevich, Maxim},
      journal={Journal of the American Mathematical Society},
      volume={31},
      number={2},
      pages={497--608},
      year={2018}
}

@eprint{Gupta2018,
      author = {{Gupta}, Meghal},
      title = "{A formula for $F$-polynomials in terms of $\bc$-vectors and stabilization of $F$-polynomials}",
      year = {2018},
      month = {12},
      archivePrefix = {arXiv},
      eprint = {1812.01910},
      primaryClass = {math.CO}
}

@article{Keller2012,
      title={Cluster algebras and derived categories}, 
      author={Bernhard Keller},
      journal={Derived Categories in Algebraic Geometry},
      year={2013},
      publisher={Euromean Mathematical Society}
}

@article{Nagao2013,
      title="{Donaldson--Thomas theory and cluster algebras}",
      volume={162},
      pages = {1313--1367},
      ISSN={0012-7094},
      DOI={10.1215/00127094-2142753},
      number={7},
      journal={Duke Mathematical Journal},
      publisher={Duke University Press},
      author={Nagao, Kentaro},
      year={2013},
      month={5}
}

@article{Nakanishi2021,
      title={Cluster algebras and scattering diagrams, Part II}, 
      author={Tomoki Nakanishi},
      year={2023},
      Journal={Cluster patterns and scattering diagrams},
      series ={MSJ Memoirs 41}
}

@article{NakanishiZelevinsky2012,
      title={On tropical dualities in cluster algebras}, 
      author={Nakanishi, Tomoki and Zelevinsky, Andrei},
      volume={565},
      pages = {217--226},
      ISSN={978-0-8218-5317-7},
      journal={Contemporary Mathmatics},
      publisher={AMS},
      year={2012}
}

@article{reading2020combinatorial,
  title={A combinatorial approach to scattering diagrams},
  author={Reading, Nathan},
  journal={Algebraic Combinatorics},
  volume={3},
  number={3},
  pages={603--636},
  year={2020}
}

\end{document}